\documentclass[10pt,reqno]{amsart}

\usepackage{amsmath,amsthm,amssymb, enumerate, color, tikz}

\usetikzlibrary{matrix,arrows,positioning,calc}

\usepackage[normalem]{ulem} 

\theoremstyle{plain}

\newtheorem{definition}{Definition}[section]
\newtheorem{theorem}[definition]{Theorem}
\newtheorem{lemma}[definition]{Lemma}
\newtheorem{corollary}[definition]{Corollary}
\newtheorem{proposition}[definition]{Proposition}

\newtheorem{remark}[definition]{Remark}

\theoremstyle{definition}
\newtheorem{example}[definition]{Example}

\numberwithin{equation}{section}

\newcommand{\A}{\mathcal A}
\newcommand{\B}{\mathcal B}
\newcommand{\C}{\mathcal C}

\newcommand{\W}{\mathcal W}

\newcommand{\mm}{\mathfrak m}

\newcommand{\Hom}{{\rm Hom}}
\newcommand{\Ext}{{\rm Ext}}
\newcommand{\Tor}{{\rm Tor}}
\newcommand{\Rmod}{R\text{-}{\rm Mod}}
\newcommand{\modR}{{\rm Mod}\text{-}R}
\newcommand{\Inj}{{\rm Inj}}
\newcommand{\Proj}{{\rm Proj}}
\newcommand{\Flat}{{\rm Flat}}

\newcommand{\ses}[3]{0 \rightarrow {#1} \rightarrow {#2} \rightarrow {#3} \rightarrow 0 } 

\newcommand{\kker}[1]{\ensuremath{{{\rm Ker}\left(#1\right)}}}

\newcommand{\FP}[1]{\mathcal{FP}_{#1}}
\newcommand{\FPinj}[1]{\mathcal{FP}_{#1}\text{-}{\rm Inj}}
\newcommand{\FPflat}[1]{\mathcal{FP}_{#1}\text{-}{\rm Flat}}
\newcommand{\Coh}[1]{{#1}\text{-$\mathit{Coh}$}}

\begin{document}

\title{Finiteness conditions and cotorsion pairs}

\author{Daniel Bravo}
 \address{
 Instituto de Ciencias F\'isicas y Matem\'aticas, 
 Universidad Austral de Chile \\
 Valdivia, Chile}
 \email{daniel.bravo@uach.cl}
 \thanks{The first author was partially supported by CONICYT/FONDECYT/Iniciaci\'on/11130324.}

 \author{Marco A. P\'erez}
 \address{Instituto de Matem\'aticas \\ Universidad Nacional Aut\'onoma de M\'exico \\ 
 Distrito Federal, M\'exico.}
 \email{maperez@im.unam.mx}
 \thanks{Part of this work was carried out during two  visits of the second author to the Universidad Austral de Chile during January and September 2015.}

\begin{abstract} \noindent We study the interplay between the notions of $n$-coherent rings and finitely $n$-presented modules, and also study the relative homological algebra associated to them. We show that the $n$-coherency of a ring is equivalent to the thickness of the class of finitely $n$-presented modules. The relative homological algebra part comes from the study of orthogonal complements to this class of modules with respect to $\Ext^1_R(F,-)$ and $\Tor_1^R(F,-)$. We also construct  cotorsion pairs from these  orthogonal complements, allowing us to provide further characterizations of $n$-coherent rings.
\end{abstract}

\maketitle

\tableofcontents

\section*{Introduction}\label{introduction}

Finitely generated and finitely presented modules (over a ring $R$) are ubiquitous in homological algebra. Many properties of modules can be described in terms of these two classes of modules, which in turn have many properties that can be stated in functorial terms. For example, the injectivity of module can be tested with respect to only finitely generated modules, and a module $M$ is finitely presented if the functor $\Hom_R(M, -)$ commutes with direct limits.

While finitely presented modules are finitely generated, the converse is not true in general. However, if $R$ is Noetherian, then these two classes coincide.  A first example of a (non Noetherian) ring where these two classes of modules do not coincide is $k[x_1,x_2,x_3\ldots]$, the ring of polynomials (over a field $k$) in countably infinite many variables. This is a well known coherent ring, i.e. a ring over which we can find finitely generated modules that are not finitely presented. From this perspective, {it seems natural to investigate certain collections of rings in order to refine the notion} {of finitely presented modules}. Recall that a finitely presented module is a finitely generated one such that it has a finite amount of relations between its finite collection of generators. Thus one can refine the class of finitely presented modules as of modules that not only have these {finiteness conditions} (on generators and on relations between generators), but also have a finite amount of relations among the relations between the generators, and then a finite amount of relations among those relations of relations, and so on {continuing up to} the $n$-th finite collection of relations among relations. This description {gives rise to notion of} \textit{finitely $n$-presented modules}. The study of the class of finitely $n$-presented modules, which we denote by $\FP{n}$, is the content Section~\ref{n-presented}, where some classic results are collected and examples exhibited.

These finiteness conditions over modules also motivate what can be thought as \textit{finiteness conditions over rings}. The notions of Noetherian and coherent rings can be generalized and stated from the point of view of these finitely $n$-presented modules, from where we get the collection of \textit{$n$-coherent rings}. It is immediate then to ask about the connection between these two concepts: finitely $n$-presented modules and $n$-coherent rings.  Section~\ref{sec:n-coherent} deals with this question and, in particular, a characterization of $n$-coherent rings in terms of finitely $n$-presented modules is established.

We also study the relative homological algebra, with respect to the class $\FP{n}$, from the injective and flat perspective; that is, modules that have \textit{a vanishing property} with respect to $\FP{n}$, and the functors $\Ext^1_R(-,-)$ and $\Tor^R_1(-,-)$. These classes of modules are called \textit{$\text{FP}_n$-injective modules} and \textit{${\rm FP}_n$-flat modules}, respectively, and denoted $\FPinj{n}$ and $\FPflat{n}$. Some of the presented results are adaptations of \cite{Ding-Chen-n-coherent} and \cite{BGH}. However, in the former reference, the authors consider slightly different notions of relative injective modules and relative flat modules, while the latter reference can be regarded as the $n = \infty$ case.  This is done in Section~\ref{relative-hom-alg}, where most of the results are presented for the case $n > 1$, since the case $n=0$ and $n=1$ are well documented in the literature (see \cite{Gobel}, \cite{Fieldhouse}, \cite{Stenstrom}, \cite{MaoDing2007Com} and \cite{MaoDing2005}, for instance). 

Section~\ref{Cotorsion pairs associated to FPninj and FPnflat} studies the completeness of certain cotorsion pairs related to the classes $\FPinj{n}$ and $\FPflat{n}$. {In the first half we study conditions} under which $\FPinj{n}$ is the left and right half of two complete cotorsion pairs. The second half is about cotorsion pairs {involving the class $\FPflat{n}$}. In the last section, we investigate conditions for when $R$ is an $n$-coherent ring in terms of the cotorsion pairs introduced in Section~\ref{Cotorsion pairs associated to FPninj and FPnflat}, and in particular, in terms of the classes $\FPinj{n}$ and $\FPflat{n}$.

Throughout this paper, {$R$ denotes an associative ring with unit}, and $\Rmod$ and {$\modR$} the categories of left and {right} $R$-modules, respectively. Unless otherwise stated, all modules will be left $R$-modules.


\section{Finiteness conditions in modules}\label{n-presented}

Let $n \geq 0$ be an integer. An $R$-module $M$ is said to be \textbf{finitely $n$-presented}, if  there is an exact sequence 
\[ 
F_n \to F_{n-1} \rightarrow \cdots \rightarrow F_1 \rightarrow F_0 \rightarrow M \rightarrow 0,
\]
where the modules $F_i$ are finitely generated and free, for every $0 \leq i \leq n$. Such exact sequence is called a \textbf{finite $n$-presentation} of $M$, and note that it is a truncated free resolution of $M$.

This way, whenever we are given a finitely $n$-presented module, then we may think of it as a finitely generated module, such that it has is a finite collection of relations between its generators, which in turn will have a finite amount of  relations among those relations, and so on all the way up to $n$.  The idea of finitely $n$-presented can be found in the literature, in particular in \cite{Bourbaki-Commutative-Algebra}, and in \cite{Brown-Cohomology-of-groups} where they are referred to as \textit{modules of type $\text{FP}_n$}. 

Denote by \textbf{$\FP{n}$} the class of all finitely $n$-presented modules. Thus $\FP{0}$ is the class of all finitely generated modules, and $\FP{1}$ is the class of all finitely presented modules. 

Hence, whenever  a finite $n$-presentation of a module is exhibited, then we know that such module is in $\FP{n}$; in turn, if we have a finite $k$-presentation of a module in $\FP{n}$, with $k\leq n$, then we can extend that particular finite $k$-presentation to a finite $n$-presentation. This can be obtained from (a generalized) Schanuel's Lemma which we record next.

\begin{lemma}[Schanuel's Lemma] \label{Schanuel}
Let 
\[
0 \to K \to F_n \to F_{n-1} \to \cdots \to F_1 \to F_0 \to M \to 0
\]
and
\[
0 \to L \to G_n \to G_{n-1} \to \cdots \to G_1 \to G_0 \to M \to 0
\]
be exact sequences in $\Rmod$, with $F_i$ and $G_i$ projective $R$-modules. Then
\[
F_0 \oplus G_1 \oplus F_{2} \oplus G_3 \oplus \cdots \cong G_0 \oplus F_1 \oplus G_{1} \oplus F_2 \oplus \cdots.
\]
Consequently, $F_i$ is finitely generated if and only if $G_i$ is, for all $0\leq i \leq n$.
\end{lemma}

Whenever a finite $n$-presentation of $M$ can be extended to a projective resolution of $M$ by finitely generated free modules, we say that $M$ is  \textbf{finitely $\infty$-presented}, and similarly denote by \textbf{$\FP{\infty}$} the class of all such modules. Modules in $\FP{\infty}$ have appeared in the literature before, in particular in \cite{bieri}, where they are also called \textit{of type $\text{FP}_{\infty}$}. The class $\FP{\infty}$ is not empty since, at least, we always have any finitely generated free module in it.

We immediately observe that every finitely $(n+1)$-presented module is finitely $n$-presented, and thus obtain the following descending chain of inclusions:
\begin{equation}\label{chain-of-FPn}
\FP{0} \supseteq \FP{1} \supseteq \cdots \supseteq \FP{n} \supseteq \FP{n+1} \supseteq \cdots \supseteq \FP{\infty},
\end{equation}
from where it immediately follows that $\bigcap_{n \geq 0} \FP{n} \supseteq \FP{\infty}$. Indeed, this last inclusion is an equality, since any finite $n$-presentation of a module $M$ in the intersection can be extended to a finite $(n+1)$-presentation, which in turn can again be extended, and so on, thus obtaining a finitely generated free resolution of $M$. 

Another application of Schanuel's Lemma is that it allows us to show when a module $M$ is in $\FP{n}$, but not in $\FP{n+1}$. All we have to do is to exhibit a finite $n$-presentation of $M$ such that its $(n+1)$-syzygy is not finitely generated. This is what we use to illustrate how the chain of inclusions \eqref{chain-of-FPn} behaves for certain rings.

\begin{example} \label{2-coherent-ring-example}
Let $k$ be a field and $R$ be the following polynomial ring:
\[
R := k[x_1,x_2,x_3,\ldots] / ( x_i x_j)_{i,j \geq 1}.
\]
We will show that the chain of inclusions in \eqref{chain-of-FPn} is strict up to $2$.

First note that, the ideal $(x_1)$ is in $\FP{0}$, but not in $\FP{1}$. This since it fits in the following short exact sequence
\[
0 \to \mm \to R \xrightarrow{f} (x_1) \to 0,
\]
where the map $R \xrightarrow{f} (x_1)$, given by the multiplication by $x_1$, is an epimorphism with infinitely generated kernel, namely $\mm = (x_1,x_2,x_3,\ldots)$. 

Next consider the quotient $R / (x_1)$, which as an $R$-module, fits in the following exact sequence:
\[
0 \to \mm \to R \xrightarrow{f_1} R \to R / (x_1) \to 0.
\]
Here the map $f_1$ is the same as above, and the map $R \to R / (x_1)$ is given by the canonical projection. This sequence shows that $R / (x_1) \in \FP{1} \setminus \FP{2}$. 

Finally, from \cite[Proposition 2.5]{BGH} (which refers precisely to this same ring), we know that if $M \in \FP{2}$, then $M$ is finitely generated and free. However, any finitely generated and free module is always in $\FP{\infty}$, hence $\FP{2} = \FP{\infty}$. This shows that \eqref{chain-of-FPn} collapses at $2$, as follows:
\[
\FP{0} \supset \FP{1} \supset \FP{2}  = \FP{n} = \FP{\infty}.
\]
\end{example}
The next example shows that \eqref{chain-of-FPn} may never collapse. 

\begin{example}  \label{strictly-infty-ring}
Let $k$ be a field and consider the polynomial ring $R$,
\[
R := {k[\ldots,x_{3},x_2,x_1,y_1,y_2,y_3,\ldots]} / {(x_{j+1} x_{j},  x_1 y_1, y_1 y_i)_{i,j \geq 1}}
\]
Then $(y_1) \in \FP{0} \setminus \FP{1}$, since the short exact  sequence, 
\[
0 \to \mm' \to R \xrightarrow{g_1} (y_1) \to 0,
\]
shows that the infinitely generated module $\mm'=(x_1,y_1,y_2,y_3,\ldots)$ is the kernel of the epimorphism $R \xrightarrow{g_1} (y_1)$, given by the multiplication by $y_1$. 

Next observe that $(x_1) \in \FP{1} \setminus \FP{2}$, since we have the following exact sequence:
\[
0 \to (x_1,x_3) \oplus \mm' \to R \oplus R \xrightarrow{f_2} R \xrightarrow{f_1} (x_1) \to 0.
\]
The map  $R \xrightarrow{f_1} (x_1)$ is an epimorphism, with a kernel generated by $x_2$ and $y_1$. This, in turn, gives a map $f_2$ from $R \oplus R$ onto these two generators, say $e_1=(1,0) \mapsto x_2$ and $e_2 = (0,1) \mapsto y_1$, providing the infinitely generated kernel $(x_1,x_3)\oplus \mm'$.

Similarly $(x_2) \in \FP{2} \setminus \FP{3}$, as we can see from the exact sequence:
\[
 \bigoplus_{1 \leq i \leq 4} R \xrightarrow{h_4} R \oplus R \xrightarrow{h_2}  R \xrightarrow{f_2} (x_2) \to 0,
\]
where $f_2$ is multiplication by $x_2$, and so $\kker{f_2} = (x_1,x_3)$. Thus, we define $h_2(e_1)=x_1$ and $h_2(e_2)=x_3$, showing that $\kker{h_2} = (y_1,x_2) \oplus (x_2,x_4)$. Now $h_4$ maps the first two generators, $e_1$ and $e_2$ to $(y_1,0)$ and $(x_2,0)$ (the next two generators, $e_3$ and $e_4$, are mapped to $(0,x_2)$ and $(0,x_4)$ respectively). So $y_i =(y_i,0,0,0) \in \kker{h_4}$, for all $i \geq 1$, making it an infinitely generated $R$-module, which is what we need.

Iterating this procedure, we observe that $(x_i) \in \FP{i} \setminus \FP{i+1}$ for $i \geq 1$. Hence in this case, the chain in \eqref{chain-of-FPn} is strict at every level.
\end{example}

The classes $\FP{n}$  and $\FP{\infty}$ has been studied before by several authors such as Bieri \cite{bieri}, Brown \cite{Brown-Cohomology-of-groups}, Bourbaki \cite{Bourbaki-Commutative-Algebra} and Glaz \cite{glaz}. We collect some results available in these references. 

Following the notation introduced Bourbaki \cite[pp. 41-42]{Bourbaki-Commutative-Algebra}, if $M$ is a finitely generated module, then we set
\[
\lambda_R(M) :=  \sup\{ n  {\geq 0} \text{ $|$ } \text{there is a finite $n$-presentation of $M$} \}.
\]
If $M$ is not finitely generated, then we set $\lambda_R(M) := -1$. When there is no ambiguity with respect to the ring we are working with, we simply use $\lambda(M)$ instead of $\lambda_R(M)$. This notation relates to the context of $\FP{n}$ as follows.

\begin{remark} \label{remark-equiv-lambda} 
For every $n \geq 0$, we have
\begin{enumerate} 
\item $M \in \FP{n}$ if and only if  $\lambda(M) \geq n$. 
\item $M \in \FP{n} \setminus \FP{n+1} $ if and only if $\lambda(M) = n$.
\item $M \in \FP{\infty}$ if and only if $\lambda(M) = \infty$.
\end{enumerate}
\end{remark}

The following result could be stated in terms of $\FP{n}$, but we keep it terms of $\lambda$ due to its simplicity.

\begin{theorem}[{\cite[Theorem 2.1.2]{glaz}}] \label{Glaz-theorem-on-FPn} 
Let $0 \rightarrow A \rightarrow B \rightarrow C \rightarrow 0$ be a short exact sequence of left $R$-modules. Then:
\begin{enumerate}

\item $\lambda(C) \geq \min\{ \lambda(B), \lambda(A) + 1 \}$. \label{Glaz-for-C}

\item $\lambda(B) \geq \min \{ \lambda(A), \lambda(C) \}$. \label{Glaz-for-B}

\item $\lambda(A) \geq \min\{ \lambda(B), \lambda(C) - 1 \}$. \label{Glaz-for-A}

\item If $B = A \oplus C$, then $\lambda(B) = \min\{ \lambda(A), \lambda(C) \}$.  \label{Glaz-for-sums}
\end{enumerate} 
\end{theorem}

Note that if the modules $B$ and $C$ in the theorem above are assumed to be in $\FP{n}$ (i.e. $\lambda(B) \geq n$ and $\lambda(C)\geq n$), then $A \in \FP{n-1}$. This means that $\FP{n}$ is not necessarily  closed under kernels of epimorphisms. However, we have the following closure results.

\begin{proposition}[Closure properties of $\FP{n}$] \label{properties_FPn} 
Let $n \geq 0$.
\begin{enumerate}

\item $\FP{n}$ is closed under cokernels of monomorphisms. \label{FPn-closed-under-cokernels}

\item $\FP{n}$ is closed under extensions. \label{FPn-closed-under-extensions}

\item $\FP{n}$ is closed under direct summands. \label{FPn-closed-under-direct-summands}
\end{enumerate} 
\end{proposition}

\begin{proof} 
Parts \eqref{FPn-closed-under-cokernels} and \eqref{FPn-closed-under-extensions} follow from Theorem \ref{Glaz-theorem-on-FPn} part \eqref{Glaz-for-C}  and part \eqref{Glaz-for-B}, respectively.

For part \eqref{FPn-closed-under-direct-summands}, suppose that $B \cong A \oplus C$ with $B \in \FP{n}$. Then $A$ and $C$ are finitely generated and from Theorem \ref{Glaz-theorem-on-FPn} part \eqref{Glaz-for-sums}, we get that $n \leq \lambda(B) \leq \lambda(A), \lambda(C)$. Therefore $A,C \in \FP{n}$.
\end{proof}

The lack of closure by kernels of epimorphisms means that $\FP{n}$ fails of being a thick class. Recall that a class of $R$-modules $\W$ is said to be \textbf{thick} if it  is closed under direct summands, and  whenever we are given a short exact sequence 
\[
0 \rightarrow A \rightarrow B \rightarrow C \rightarrow 0
\]
with two out of the three terms $A$, $B$, $C$ in $\W$, then so is the third. 

It follows that $\FP{n}$ is thick if, and only if, it is closed under  kernels of epimorphisms. The question on the thickness of $\FP{n}$ then reduces to knowing under which conditions $\FP{n}$ is closed under kernels of epimorphisms in $\FP{n}$. This will be settled in Section~\ref{sec:n-coherent}. We begin by noting that the class $\FP{\infty}$ is  always thick.

\begin{theorem} \label{FPinfty-is-thick}
For any ring the class $\FP{\infty}$ is thick.
\end{theorem}

\begin{proof} 
The equality $\FP{\infty} = \bigcap_{n \geq 0} \FP{n}$, along with Proposition~\ref{properties_FPn} gives us that $\FP{\infty}$ is closed under direct summands, extensions, and cokernels of monomorphisms. 

Next, consider a short exact sequence $0 \rightarrow A \rightarrow B \rightarrow C \rightarrow 0$ and suppose $B, C \in \FP{\infty}$ (i.e., $\lambda(B) = \infty$ and $\lambda(C) = \infty$), then part \eqref{Glaz-for-A} of Theorem \ref{Glaz-theorem-on-FPn} implies that $\lambda(A) \geq \min\{ \lambda(B),\lambda(C)-1 \} = \infty$. This says that $A \in \FP{\infty}$. 
\end{proof}

To end this section we recall some results describing the relation between the classes $\FP{n}$ and $\FP{\infty}$, and the derived functors to $\Hom_R(-,-)$ and $-\otimes_R -$.

\begin{theorem}[{\cite[Theorem 2]{Brown-Homological}}, {\cite[Theorem 1.3]{bieri}}] \label{FPn_Ext_Tor}
The following conditions are equivalent for every right $R$-module $M$ and every $n \geq 0$:
\begin{enumerate}
\item $M \in \FP{n+1}$.
\item $\Ext^i_R(M,-)$ commutes with direct limits for all $0 \leq i \leq n$.
\item $\Tor_i^R(M,-)$ commutes with direct products for all $0 \leq i \leq n$. 
\end{enumerate} 
\end{theorem}

Note that the well known result of every finitely presented module $M$ commuting with direct limits under $\Hom_R(M,-)$, in terms of finitely $n$-presented modules, is the case $n=0$ in the previous theorem. The version for $\FP{\infty}$ reads as follows:

\begin{theorem}[{\cite[Corollary 1.6]{bieri}}]
The following statements are equivalent for every right $R$-module $M$:
\begin{enumerate}
\item $M \in \FP{\infty}$.
\item $\Ext^i_R(M,-)$ commutes with direct limits for all $i \geq 0$.
\item $\Tor^R_i(M,-)$ commutes with direct products for all $i \geq 0$.
\end{enumerate}
\end{theorem}


\section{$n$-Coherent rings}\label{sec:n-coherent}

Recall that a ring $R$ said to be (left) \textit{coherent} if every finitely generated (left) ideal of $R$ also is finitely presented. Another equivalent definition for coherent ring reads as follows: $R$ is coherent if, and only if, every finitely generated submodule of a free $R$-module is also finitely presented. A similar result can be stated for Noetherian rings, namely that over a Noetherian ring, any submodule of a finitely generated free module is finitely generated. In terms of finitely $n$-presented modules, we have that a module is  $1$-presented if, and only if, it is finitely presented. Therefore, we can expect an equivalence for coherent rings in terms of finitely $n$-presented modules; we record this as the following remark.

\begin{remark} 
A ring $R$ is coherent if, and only if, every module in $\FP{1}$ is also in $\FP{2}$. 
\end{remark}

 Following these observations, it seems natural to present the next definition.

\begin{definition}\label{def-n-coherent} 
A ring $R$ is \textbf{$n$-coherent} if $\FP{n} \subseteq \FP{n+1}$.
\end{definition}

This way a coherent ring is a $1$-coherent ring, and a Noetherian ring is a $0$-coherent ring. An example of a $2$-coherent ring was given in Example \ref{2-coherent-ring-example}. 

The idea of $n$-coherent rings has been studied before, and it seems to have been first introduced (as shown above) in the literature by D. L. Costa \cite{Costa}. However, there are several variations of the definitions of $n$-coherent rings. For example, D. E. Dobbs, S. E. Kabbaj, and N. Mahdou \cite{dobbs}, work with what they call \emph{strong $n$-coherent rings}, and it is this definition that agrees with Definition~\ref{def-n-coherent}. Other authors, such as L. Mao and N. Ding \cite{MaoDing2007}, introduce an additional  parameter; this way $R$ is $n$-coherent (as shown in Definition \ref{def-n-coherent}) if, and only if, $R$ is \emph{$(n,\infty)$-coherent} (as in  \cite{MaoDing2007}). In all cases there seem to be some intersection with Definition \ref{def-n-coherent}.


We immediately note from Definition \ref{def-n-coherent} that an $n$-coherent ring, is also a $k$-coherent for all $k \geq n$. By convention, any ring  is $\infty$-coherent (this is motivated by the naive observation that ``$\FP{\infty} \subseteq \FP{\infty+1}$"). Thus similarly to the chain of inclusion \eqref{chain-of-FPn}, we obtain an strictly ascending  chain of inclusions of classes of rings: 
\begin{equation} \label{chain-for-n-coherent}
\Coh{0} \subset \Coh{1} \subset \Coh{2} \subset \cdots \subset \Coh{n} \subset \cdots \Coh{\infty},
\end{equation}
where $\Coh{n}$ denotes the class of all $n$-coherent rings. 

The chain above suggests a connection between a ring being $n$-coherent, and the class of finitely $n$-presented $R$-modules. Indeed, using our terminology we have the following.

\begin{remark} The following results are well known:
\begin{enumerate}
\item $R \in \Coh{0}$ if and only if the class $\FP{0}$ is thick.
\item $R \in \Coh{1}$ if and only if the class $\FP{1}$ is thick.
\item For any ring $R$ (i.e. for any $R \in \Coh{\infty}$) the class $\FP{\infty}$ is thick (see Theorem \ref{FPinfty-is-thick}).
\end{enumerate}
\end{remark}

The following theorem settles the conditions for the cases in between.

\begin{theorem}\label{characterization-of-n-coherent} 
Let $R$ be a ring and $n \geq 0$. The following are equivalent.
\begin{enumerate}
\item $R$ is $n$-coherent. \label{R-n-coh}

\item $\FP{n}$ is thick. \label{FPn-thick}

\item $\FP{n} = \FP{\infty}$. \label{FPn=FPinfty} 
\end{enumerate} 
\end{theorem}

\begin{proof} \

\eqref{FPn=FPinfty} $\Rightarrow$ \eqref{FPn-thick}. This is immediate from Theorem \ref{FPinfty-is-thick}.  

\eqref{FPn-thick} $\Rightarrow$ \eqref{R-n-coh}. Let $M \in \FP{n}$. Then there is an exact sequence
\[
F_n \to F_{n-1} \rightarrow \cdots \rightarrow F_2 \xrightarrow{} F_1 \rightarrow F_0 \xrightarrow{f_0} M \rightarrow 0,
\]
from which we obtain the exact sequence $\ses{\kker{f_0}}{F_0}{M}$. Since, by hypothesis $\FP{n}$ is thick, and $F_0$ and $M$ are both in $\FP{n}$, then $\kker{f_0} \in \FP{n}$. Thus $\kker{f_0}$ has a finite $n$-presentation, which along with the previous short exact sequence, gives a finite $(n+1)$-presentation of $M$. This means that $M \in \FP{n+1}$.

\eqref{R-n-coh} $\Rightarrow$ \eqref{FPn=FPinfty}. Let $M \in \FP{n}$, then similarly as above and using the first syzygy of $M$ (occurring in a finite $n$-presentation of $M$), we obtain a finite $(n+1)$-presentation of $M$, i.e. $M \in \FP{n+1}$. Next apply the same process to the second syzygy of $M$ to obtain that $M \in \FP{n+2}$. Iterating this procedure yields a finite $k$-presentation of $M$ for all $k \geq n$. Hence $M \in \bigcap_{n \geq 0} \FP{n} = \FP{\infty}$.
\end{proof}

\begin{remark}
The ring in Example \ref{strictly-infty-ring} is not $n$-coherent for any $n \geq 0$. Since if we suppose if $k$-coherent for some $k \geq 0$, then the chain in \eqref{chain-of-FPn} would collapse at $k$, but we established that this can not happen. We call rings with this non collapsing property, \textbf{strictly $\infty$-coherent} rings.
\end{remark}

Now we are able to provide a more precise result indicating how the $n$-coherency of a ring relates to the the class $\FP{n}$, by combining Proposition \ref{properties_FPn} and Theorem \ref{characterization-of-n-coherent}.

\begin{corollary}\label{FPn-closed-kernel}
A ring $R$ is $n$-coherent if, and only if, the class $\FP{n}$ is closed by kernels of epimorphisms. 
\end{corollary}


Note that if $R$ is Noetherian, then $\FP{0} = \FP{1}$, and so the notion of \textit{absolutely pure} (or \textit{FP-injective}) modules, introduced by Maddox \cite{maddox} and  Stenst\"{o}rm \cite{Stenstrom}, and the notion of injective modules agree. This is not the case, when $R$ is coherent, as these two classes do not coincide. So it seems natural to explore $R$-modules which are injective with respecto to $\FP{n}$, but for any ring $R$. This is what we do in the next section. 


\section{$\text{FP}_n$-injective and ${\rm FP}_n$-flat modules} \label{relative-hom-alg}

The class of injective modules over a ring $R$ can be defined as the collection of all $R$-modules $N$, such that $\Ext^1_R(M, N) = 0$ for all finitely generated modules $M$, i.e. for all $M \in \FP{0}$. 

Similarly,  Stenstr\"{o}m \cite{Stenstrom} defines the class of \textbf{FP-injective} as those modules $M$ such that $\Ext^1_R(F,M) = 0$ for all finitely presented modules $F$, i.e. for all $F \in \FP{1}$. Recently in \cite{BGH}, the class of \textbf{absolutely clean} modules where defined as those modules $M$ such that $\Ext^1_R(F,M) = 0$ for all $F \in \FP{\infty}$. Motivated by this, we present the following definition.

\begin{definition} \label{FPn-injective} 
We say that a left $R$-module $M$ is \textbf{$\text{FP}_n$-injective} if $\Ext^1_R(F, M) = 0$ for all $F \in \FP{n}$ (this may include the case $n = \infty$). We denote by $\FPinj{n}$ the class of all $\text{FP}_n$-injective modules. 
\end{definition}

With this definition, $M$ is injective if, and only if, $M$ is $\text{FP}_0$-injective, and $M$ is FP-injective if, and only if, $M$ is $FP_1$-injective. The case of $\text{FP}_{\infty}$-injective modules (i.e. absolutely clean) is the same as introduced in \cite{BGH}. We also observe that this definition of $\text{FP}_n$-injective modules differs from that of J. Chen and N. Ding \cite{Ding-Chen-n-coherent} for $n>1$ (where they consider orthogonality with respect to $\Ext^n_R(-,-)$ instead). 

We continue with a definition of flatness with respect to the class $\FP{n}$.

\begin{definition}\label{FPn-flat} 
We say that a left $R$-module $M$ is \textbf{\text{$\text{FP}_n$-flat}} if $\Tor_1^R(F, M) = 0$ for all $F \in \FP{n}$ (this may include the case $n = \infty$). We denote by $\FPflat{n}$ the class of all $\text{FP}_n$-flat modules.
\end{definition}

Note that flat modules coincide with the FP$_0$-flat modules, and \textit{level modules}, in the sense of \cite{BGH}, coincide with the FP$_{\infty}$-flat modules, i.e. those $M \in \Rmod$ for which $\Tor_1^R(F, M) = 0$ for all $F \in \FP{\infty}$. 

From the descending chain of inclusions \eqref{chain-of-FPn}, we get the following ascending chains of inclusions: 
\begin{equation}\label{chain-of-FPn-inj}
\FPinj{0} \subseteq \FPinj{1} \subseteq \cdots \subseteq \FPinj{n} \subseteq \cdots \subseteq \FPinj{\infty}
\end{equation}
and 
\begin{equation}\label{chain-of-FPn-flat}
\FPflat{0} \subseteq \FPflat{1} \subseteq \cdots \subseteq \FPflat{n} \subseteq \cdots \subseteq \FPflat{\infty}.
\end{equation}

We immediately note that for $n \geq 0$, the class of injective modules is contained in $\FPinj{n}$, and that the class of flat modules is contained in $\FPflat{n}$.

\begin{remark} 
As we have previously mentioned, the case $n=0$ is the study of injective modules; the case $n=1$ is the case of absolutely pure modules, or FP-injective modules, which has been studied in \cite{Stenstrom} and \cite{Ding-Chen-n-coherent}. Thus for the rest of this article, we only focus on the cases $n>1$. 
\end{remark}

 
We  study how the classes $\FPinj{n}$ and $\FPflat{n}$ relate via the notion of character modules. Recall that the \textbf{character module} or \textbf{Pontryagin dual module} of a left $R$-module $M$ is defined as the right $R$-module $M^+ := \Hom_{\mathbb{Z}}(M,\mathbb{Q/Z})$. Observe that $\mathbb{Q / Z}$ is an injective cogenerator of $\Rmod$ (see \cite{Borceaux}) which among other properties gives us that $M=0$, whenever $M^+=0$. 
The following results are standard and included here for further reference. The reader can see the proof in \cite[Lemma 1.2.11]{Gobel}.


\begin{theorem}[Ext-Tor relations]\label{Tor-Ext-relations}
Let $R$ and $S$ be rings. 
\begin{enumerate}
\item \label{Ext(M-N-dual)-Tor(M-N)-dual} Let $M$ be a right $R$-module, and $N$ be an $(S,R)$-bimodule. If $I$ is an injective left $S$-module, then for all $i \geq 0$ 
\[
\Ext^i_R(M,\Hom_S(N,I)) \cong \Hom_S(\Tor_i^R(M,N),I).
\]
In particular,
\[
\Ext^1_R(M,N^+) \cong \Tor_1^R(M,N)^+.
\]


\item \label{Tor(F-N-dual)-Ext(F-N)-dual} If $F$ is right $R$-module in $\FP{n+1}$, $N$ is an $(S,R)$-bimodule, and $M$ is an injective left $S$-module, then for each $1 \leq i \leq n$ 
\[
\Tor^R_i(F, \Hom_S(N,M)) \cong \Hom_S(\Ext^i_R(F,N),M).
\] 
In particular, 
\[
\Tor^R_1(F,N^+) \cong \Ext^1_R(F,N)^+.
\]
\end{enumerate}
\end{theorem}

The duality results between $\FPinj{n}$ and $\FPflat{n}$ are as follows.

\begin{proposition} \label{FPnflat-iff-FPninj^+} 
Let $M \in \Rmod$ and $n>1$. Then, $M \in \FPflat{n}$ if and only if $M^+  \in \FPinj{n}$. 
\end{proposition}

\begin{proof} From part \eqref{Ext(M-N-dual)-Tor(M-N)-dual} of Theorem \ref{Tor-Ext-relations}, we have $\Tor^R_1(F,M)^+  \cong \Ext^1_R(F,M^+)$ for any right $R$-module $F$. In particular, for $F \in \FP{n}$.

If $M \in \FPflat{n}$, then $\Tor^R_1(F,M) = 0$, and so $\Ext^1_R(F,M^+) = 0$. Hence $M^+ \in \FPinj{n}$.  

If $M^+ \in \FPinj{n}$, then $\Ext^1_R(F,M^+) = 0$, and hence $\Tor^R_1(F,M)^+ = 0$. This implies that $\Tor^R_1(F,M) = 0$,  and so $M \in \FPflat{n}$.
\end{proof}

\begin{proposition}\label{FPninj-iff-FPflat^+} 
Let $M \in \modR$ and $n>1$. Then, $M^+ \in \FPflat{n}$ if and only if $M  \in \FPinj{n}$.
\end{proposition}

\begin{proof} From part \eqref{Tor(F-N-dual)-Ext(F-N)-dual} of Theorem \ref{Tor-Ext-relations}, we have $\Ext^1_R(F,M)^+  \cong \Tor_1^R(F, M^+)$ for any $R$-module $F$. In particular, for $F \in \FP{n}$.

If $M \in \FPinj{n}$, then $\Ext^1_R(F,M) = 0$, and so $\Tor^R_1(F,M^+) = 0$. Hence $M^+ \in \FPflat{n}$.

If $M^+ \in \FPflat{n}$, then $\Tor^R_1(F,M^+) = 0$, and so $\Ext^1_R(F,M)^+ = 0$.  This implies that  $\Ext^1_R(F,M)=0$, and so $M \in \FPinj{n}$.
\end{proof}

Note that the version for $\FPinj{\infty}$ y  $\FPflat{\infty}$ are already done in  \cite[Theorem 2.13]{BGH}, in the language of absolute clean and level modules.

An immediate corollary of this proposition is an application to duality pairs \cite{BGH}. Given a ring $R$, and two classes of module $\A$ and $\B$, then we say that $(\A,\B)$ is a \textbf{duality pair} if and only if $\A^+ \subseteq \B$ and $\A \supseteq \B^+$; that is, $M \in \A$ if and only if $M^+ \in \B$, and $N \in \B$ if and only if $N^+ \in \A$.

\begin{corollary} \label{duality-pair}
Let $n>1$. The pair $(\FPflat{n},\FPinj{n})$ is a duality pair.
\end{corollary}

Recall that  $M^{++}$  denotes the double dual of a module $M$. We also have the following result.

\begin{corollary}\label{FPninj-and-FPnflat-double-dual}
Let $n > 1$.
\begin{enumerate}
\item $M \in \FPinj{n}$ if and only if $M^{++} \in \FPinj{n}$.
\item $M \in \FPflat{n}$ if and only if $M^{++} \in \FPflat{n}$.
\end{enumerate}
\end{corollary}

In the next results, we will need the following closure properties about direct summands.

\begin{proposition} \label{closed-under-direct-summands} 
The classes $\FPinj{n}$ and $\FPflat{n}$ are closed under direct summands.
\end{proposition}

\begin{proof} 
We only prove that $\FPinj{n}$ is closed under direct summands. The arguments will carry over to $\FPflat{n}$. 

Let $M \in \FPinj{n}$ and $N$ be a direct summand of $M$. Then $M = N \oplus N'$ for some $N \in \Rmod$. So for every $F \in \FP{n}$, we have $0 = \Ext^1_R(F,M) \cong \Ext^1_R(F,N) \oplus \Ext^1_R(F,N')$. Hence we obtain $\Ext^1_R(F,N) = 0$ for every $F \in \FP{n}$, i.e. $N \in \FPinj{n}$. 
\end{proof}

The following two propositions summarize some properties of the classes $\FPinj{n}$ and $\FPflat{n}$. 

\begin{proposition}[Properties of $\FPinj{n}$] \label{properties-FPninj} 
Let $n>1$, then the following conditions hold:
\begin{enumerate}
\item $\FPinj{n}$ is closed under extensions. \label{FPninj-closed-extensions}

\item $\FPinj{n}$ is closed under direct products. \label{FPninj-closed-products}

\item $\FPinj{n}$ is closed direct limits. \label{FPninj-closed-direct-limits}

\item $\FPinj{n}$ is closed under pure submodules and under pure quotients. \label{FPninj-closed-pure}

\end{enumerate}

\end{proposition}

\begin{proof} Part \eqref{FPninj-closed-extensions} follows immediately by the long exact sequence of $\Ext^1_R(F,-)$ for any $F \in \Rmod$. Since $\Ext^1_R(F,-)$ preserves limits, part \eqref{FPninj-closed-products} also follows. 

Next, to see \eqref{FPninj-closed-direct-limits},  let $F \in \FP{n}$ and $n > 1$. Then, by Theorem \ref{FPn_Ext_Tor}, we know $\Ext^1_R(F,-)$ commutes with direct limits, giving the result.

To show \eqref{FPninj-closed-pure} suppose we are given a pure exact sequence 
\begin{equation} \label{pure-ses}
0 \to N \to M \to M / N \to 0
\end{equation}
with $M \in \FPinj{n}$. Then we get the following sequence 
\begin{equation} \label{induce-ses}
0 \to \Hom_R(F,N) \to \Hom_R(F,M) \to \Hom_R(F,M / N) \to 0
\end{equation}
which is exact for every finitely presented  module $F$, and in particular for every $F \in \FP{n}$ when $n>1$. On the other hand, we have an induced exact sequence
\[
\Hom_R(F,M) \to \Hom_R(F, M / N) \to \Ext^1_R(F,N) \to \Ext^1_R(F,M).
\]
The map $\Hom_R(F,M) \to \Hom_R(F, M / N)$ is an epimorphism since (\ref{induce-ses}) is exact, and $\Ext^1_R(F,M) = 0$ since $M \in \FPinj{n}$. Therefore $\Ext^1_R(F,N) = 0$ for every $F \in \FP{n}$, i.e. $N \in \FPinj{n}$. 

To show that $M / N \in \FPinj{n}$, observe that since the sequence in \eqref{pure-ses} is exact, then by \cite[Lemma 1.2.13 (d)]{Gobel} we have a split exact sequence 
\[
0 \to (M / N)^+ \to M^+ \to N^+ \to 0 .
\] 
Then $(M / N)^+$ is a direct summand of $M^+$. On the other hand, by Proposition \ref{FPninj-iff-FPflat^+}), $M^+ \in \FPflat{n}$. Thus $(M / N)^+ \in \FPflat{n}$ since the class $\FPflat{n}$ is closed under direct summands by Proposition \ref{closed-under-direct-summands}. Then by  Proposition \ref{FPninj-iff-FPflat^+} again, we have that $M / N \in \FPinj{n}$. 
\end{proof}

\begin{proposition}[Properties of $\FPflat{n}$] \label{properties-FPnflat} 
Let $n>1$, then the following conditions hold:
\begin{enumerate}
\item $\FPflat{n}$ is closed under extensions. \label{FPnflat-closed-extension}

\item $\FPflat{n}$ is closed under direct limits. \label{FPnflat-closed-direct-limits}

\item $\FPflat{n}$ is closed under direct products. \label{FPnflat-closed-direct-products}

\item $\FPflat{n}$ is closed under pure submodules and pure quotients. \label{FPnflat-closed-pure}

\end{enumerate}

\end{proposition}

\begin{proof} 
Part \eqref{FPnflat-closed-extension} follows directly from the long exact sequence of $\Tor_1^R(F,-)$. Part \eqref{FPnflat-closed-direct-limits}  also follows since $\Tor_1^R(F,-)$ preserves colimits.

To show \eqref{FPnflat-closed-direct-products}, use Theorem \ref{FPn_Ext_Tor} with $F \in \FP{n}$ and $n > 1$. Then, $\Tor_1^R(F,-)$ will commute with direct products, and the result follows. 

Finally, let $M \in \FPflat{n}$ and suppose we have pure short exact sequence 
\[
\ses{N}{M}{M / N}.
\] 
Then the following short exact sequence splits
\[
\ses{(M / N)^+}{M^+}{N^+}.
\] 
Since $M^+ \in \FPinj{n}$ (by Proposition \ref{FPnflat-iff-FPninj^+}), and $\FPinj{n}$ is closed under summands (by Proposition \ref{closed-under-direct-summands}), both $N^+$ and $(M/N)^+$ are in $\FPinj{n}$. Then by the same Proposition \ref{FPnflat-iff-FPninj^+} again, the result follows.
\end{proof}


\section{Cotorsion pairs associated to $\text{FP}_n$-injective and $\text{FP}_n$-flat modules} \label{Cotorsion pairs associated to FPninj and FPnflat}

We now connect with the notion of cotorsion pairs and study the associated pairs to the classes $\FPinj{n}$ and $\FPflat{n}$ with, and without, conditions on the ring. 

For every class $\C$ of $R$-modules, denote the classes
\begin{align*}
\C^{\perp} & := \{ X \in \Rmod \mbox{ : } \Ext^1_R(C, X) = 0, \text{ for all $C \in \C$} \}, \mbox{ and} \\
{}^{\perp}\C & := \{ X \in \Rmod \mbox{ : } \Ext^1_R(X,C) = 0, \text{ for all $C \in \C$} \},
\end{align*} 
which are referred to as the \textbf{right} and \textbf{left orthogonal complements} of $\C$.

Given two classes of $R$-modules $\A$ and $\B$, then we say that $(\A,\B)$ is a \textbf{cotorsion pair} if $\A^{\perp} = \B$ and $\A = {}^{\perp}\B$. Two straightforward examples of cotorsion pairs are the pairs $(\Proj,\Rmod)$ and $(\Rmod, \Inj)$, where $\Proj$ denotes the class of projective $R$-modules, and $\Inj$ denotes the class of injective $R$-modules. A less trivial example is the pair $(\Flat,\text{Cotorsion})$, where $\Flat$ denotes the class of flat $R$-modules, and Cotorsion, is its right orthogonal complement, the class of the \textit{cotorsion modules} (see \cite[Proposition 7.4.3]{EJ}).

A cotorsion pair $(\A,\B)$ in $\Rmod$ is said to be \textbf{complete} if for any $X \in \Rmod$ we can always find short exact sequences 
\[
\ses{B}{A}{X} \mbox{ \ and \ } \ses{X}{B'}{A'}
\]
with $A,A' \in \A$ and $B, B' \in \B$. Furthermore, a cotorsion pair is called \textbf{hereditary} if $\Ext^i_R(A,B) = 0$ for all $i>0$, $A \in \A$ and $B \in \B$. Knowing that the cotorsion pair $(\A,\B)$ is hereditary, is equivalent to knowing that $\A$ is  \textbf{resolving} (that is, $\Proj \subseteq \A$, and $\A$ is closed under extensions and kernels of epimorphisms), which is also equivalent to knowing that $\B$ is \textbf{coresolving} (that is, $\Inj \subseteq \B$, and $\B$ is closed under extensions and cokernels of monomorphisms). 

Since from now on, our focus is on the cases $n>1$, we recall some cotorsion pairs in the first two cases. For $n = 0$,  the class $\FPinj{0}$ coincides with the class $\Inj$ of injective modules, and we know that $(\Rmod,\Inj)$ is a complete cotorsion pair for every ring $R$. On the other hand, it is known that $(\Inj,\Rmod)$ is a perfect cotorsion pair if, an only if, $R$ is a self injective Noetherian ring (see \cite[Theorem 5.4.1]{EJ} and \cite[Theorem 4.1.13]{Gobel}). In the $FP_0$-flat version, we have that $\Flat = \FPflat{0}$, and the cotorsion pair $(\FPflat{0},(\FPflat{0})^{\perp})$ indicated above is complete.

The injective version for the case $n = 1$ can be found in several sources. For every ring $R$, we have the complete cotorsion pair $(\text{}^\perp(\FPinj{1}), \FPinj{1})$ \cite[Theorem 4.1.6]{Gobel}. More properties of this pair were studied by L. Mao and N. Ding in \cite{MaoDing2005} and \cite{MaoDing2007}, where they establish that, if $R$ is a left coherent ring, then $(\text{}^\perp(\FPinj{1}), \FPinj{1})$ is also hereditary. If in addition, $R$ is self FP-injective (i.e. self FP$_1$-injective) then $(\FPinj{1}, (\FPinj{1})^\perp)$ is a perfect cotorsion pair (see \cite[Theorem 3.8]{MaoDing2005} and \cite[Theorem 3.4]{MaoDing2007}). 

We do not need to impose any condition on $R$ in order to prove that $\FPinj{n}$ is the right half of a complete cotorsion pair $(\text{}^\perp(\FPinj{n}), \FPinj{n})$ (for every $n \geq 0$ and $n = \infty$). This will be a consequence of the following more general result.

\begin{proposition} \label{cotorsion-pair-class-of-fin-gen-on-right} 
Let $\mathcal{A}$ be a class $R$-modules such that every $A \in \mathcal{A}$ is finitely generated. Then the cotorsion pair $(^{\perp}(\mathcal{A}^{\perp}), \mathcal{A}^\perp)$ is cogenerated by a set, meaning that there exists a set $\mathcal{S}$ such that $\A^\perp = \mathcal{S}^\perp$. 
\end{proposition}

\begin{proof} If $M$ is a finitely generated module, then ${\rm Card}(M) \leq \max\{ \aleph_0, {\rm Card}(R) \}$. This allows us to choose a set $\mathcal{S}$ of representatives of (finitely generated) modules in $\mathcal{A}$, in such a way that every module in $\A$ is isomorphic to a module in $\mathcal{S}$. Then, by \cite[Theorem 3.2.1]{Gobel}, $(^{\perp}(\mathcal{A}^{\perp}), \mathcal{A}^\perp)$ is cogenerated by $\mathcal{S}$. 
\end{proof}

The following result is an immediate consequence of the previous proposition and the Eklof and Trlifaj Theorem (which states that every cotorsion pair cogenerated by a set is complete \cite{EklofTrlifaj}).

\begin{corollary} \label{cotorsion-pairs-FPninj-right} For any ring $R$:
\begin{enumerate}
\item $(^{\perp}(\FP{n}), \FPinj{n})$ is a complete cotorsion pair in $\Rmod$, for every $n \geq 0$.

\item $(^{\perp}(\FP{\infty}), \FPinj{\infty})$ is a complete cotorsion pair in $\Rmod$. 

\end{enumerate} 
\end{corollary}

Having obtained cotorsion pairs where $\FPinj{n}$ is on the right slot, we study conditions under which $\FPinj{n}$ is the left half of a complete cotorsion pair.

As mentioned before, in \cite[Theorem 3.4]{MaoDing2007Com} the authors prove that if $R$ is a self FP-injective left coherent ring, then $(\FPinj{1}, (\FPinj{1})^\perp)$ is a perfect cotorsion pair in $\Rmod$. We could use arguments similar to theirs to show the next theorem. However, we provide a simpler proof by using a result from  H. Holm and P. J\o rgensen, that states the following:

\begin{theorem}[{\cite[Theorem 3.4]{Holm}}] \label{Holm-and-Jorgensen-Theorem}
If $\mathcal{F}$ is a class of modules containing the ground ring $R$ and is closed under extensions, direct sums, pure submodules, and pure quotients, then $(\mathcal{F}, \mathcal{F}^\perp)$ is a perfect cotorsion pair in $\Rmod$.

In particular, $\mathcal{F}$ is covering and $\mathcal{F}^{\perp}$ is enveloping.
\end{theorem}

Recall that a class $\C$ in $\Rmod$ is said to be \textbf{covering} provided that for each module $M \in \Rmod$ there is a $\C$-cover. A map $f \in \Hom_R(C,M)$ is called a $\C$-cover of $M$ if $C \in \C$, the sequence $\Hom_R(C',C) \to \Hom(C',M) \to 0$ is exact for every $C' \in \C$, and the equality $f \circ g = f$ implies $g$ is an automorphism of $C$.

Dually, a class $\C$ in $\Rmod$ is said to be \textbf{enveloping} provided that for each module $M \in \Rmod$ there is a $\C$-envelope. A map $f \in \Hom_R(M,C)$ is called a $\C$-envelope of $M$ if $C \in \C$, the sequence $\Hom_R(C,C') \to \Hom(M,C') \to 0$ is exact for every $C' \in \C$, and the equality $f = f \circ g$ implies $g$ is an automorphism of $C$.

Finally, we say that a cotorsion pair $(\A,\B)$ is called \textbf{perfect} if $\A$ is covering and $\B$ is enveloping.   Now we are ready to state the conditions under which $\FPinj{n}$ and $\FPinj{\infty}$ are the left half of a cotorsion pair.

\begin{theorem}\label{cotorsion-pairs-FPninj-left} Let $n>1$.
\begin{enumerate}
\item If $R$ is a self FP$_{n}$-injective ring (i.e. $R \in \FPinj{n}$), then there is a perfect cotorsion pair $(\FPinj{n}, (\FPinj{n})^\perp)$ in $\Rmod$. 

\item If $R$ is a self FP$_{\infty}$-injective ring (i.e. $R \in \FPinj{\infty}$), then there is a perfect cotorsion pair $(\FPinj{\infty}, (\FPinj{\infty})^\perp)$ in $\Rmod$. 
\end{enumerate} 
\end{theorem}

\begin{proof} 
In the first statement, note that the hypothesis $R \in \FPinj{R}$ and Proposition~\ref{properties-FPninj} gives all the properties required in Theorem~\ref{Holm-and-Jorgensen-Theorem}.  Therefore we get the  perfect cotorsion pair $(\FPinj{n}, (\FPinj{n})^\perp)$. 

The second statement follows similarly by the hypothesis $R \in \FPinj{\infty}$ and \cite[Proposition 2.7]{BGH}.
\end{proof}

We conclude this section by showing that the class $\FPflat{n}$ is also the left half of a perfect cotorsion pair, for every ring $R$. We use the result from \cite{Holm} again for that purpose. Observe that in  Theorem 2.14 from \cite{BGH}, the authors prove that the class $\FPflat{\infty}$ of level modules, is the left half of a perfect cotorsion pair.

\begin{theorem} \label{cotorsion-pair-FPnflat-left} The class $\FPflat{n}$ of ${\rm FP}_n$-flat modules is the left half of a perfect cotorsion pair $(\FPflat{n}, (\FPflat{n})^\perp)$.
\end{theorem}

\begin{proof} First, note that $R \in \FPflat{n}$ for every $n \geq 0$.  The rest follows by Proposition \ref{properties-FPnflat} and Theorem \ref{Holm-and-Jorgensen-Theorem}.
\end{proof}

\begin{remark} \ 
\begin{enumerate}
\item In analogy with both cotorsion pairs for the class $\FPinj{n}$, we would like to have a cotorsion pair such that the class $\FPflat{n}$ is in the right side of the pair; however, this goal, so far, seems elusive. Part of this, is that it is not clear what hypothesis are needed, since unlike the version with $\FPinj{n}$ on the left, the ground ring is always in $\FPflat{n}$.
\item Theorem \ref{cotorsion-pairs-FPninj-left} can be extended  to the case $n = 1$, by requiring the extra hypothesis of $R$ being coherent \cite[Theorem 3.4]{MaoDing2007Com}. 
\end{enumerate}
\end{remark}

This remark raises some questions about the implication of the $n$-coherency of the ground ring over the cotorsion pairs in Corollary~\ref{cotorsion-pairs-FPninj-right} and Theorem~\ref{cotorsion-pair-FPnflat-left}. This is what we study in the next section.


\section{Characterization of $n$-coherency via cotorsion pairs}

We begin by observing that if $F \in \FP{n}$, then $\Ext^1_R(F,M)=0$ for all $M \in \FPinj{n}$. We would like to state the reciprocal also, i.e. if a module $F$ is such that $\Ext^1_R(F,M)=0$ for all $M \in \FPinj{n}$, then $F \in \FP{n}$. This is known to be true in the case $n = 1$, shown by S. Glaz, in the form we state below.

\begin{theorem}[{\cite[Theorem 2.1.10]{glaz}}] \label{Glaz-theorem-equivalence-fin-presented}
Let $R$ be a ring and let $M$ be a finitely generated module satisfying $\Ext^1_R(M,N)=0$ for all $N \in \FPinj{1}$, then $M \in \FP{1}$.
\end{theorem}

If in the the previous result we replace the part of finitely generated by $\FP{n-1}$, then we obtain the following generalization.

\begin{lemma} \label{characterization-of-FPn-by-Ext1} 
Let $n \geq 1$, then the following conditions are equivalent for every $M \in \Rmod$.
\begin{enumerate}
\item $M \in \FP{n-1}$ and $\Ext^1_R(M,N) = 0$ for every $N \in \FPinj{n}$.  \label{M-is-in-FP(n-1)-and-Ext1-is-0}

\item $M \in \FP{n}$. \label{M-is-in-FPn}
\end{enumerate}
\end{lemma}

\begin{proof} 
We begin by assuming \eqref{M-is-in-FP(n-1)-and-Ext1-is-0} in order to show \eqref{M-is-in-FPn}, and use induction on $n \geq 1$. The case $n = 1$ is Theorem \ref{Glaz-theorem-equivalence-fin-presented}. Suppose the result is true for $n-1$. That is
	$M \in \FP{n-2}$ and $\Ext^1_R(M,N) = 0$ for every $N \in \FPinj{n-1}$, implies $M \in \FP{n-1}$.

Consider  $M \in \FP{n-1}$ satisfying $\Ext^1_R(M,N) = 0$ for every $N \in \FPinj{n}$. Consider a finite $(n-1)$-presentation of $M$, say
\[
F_{n-1} \to F_{n-2} \to \cdots \to F_{1} \to F_0 \stackrel{f_0}\to M \to 0
\]
If we show that $K = \kker{f_0} \in \FP{n-1}$, then we are done, since this will produce a finite $n$-presentation of $M$. To do this we apply the induction hypothesis stated above, since we already have that $K \in \FP{n-2}$. Thus we only have to show that $\Ext^1_R(K, N) = 0$ for every $N \in \FPinj{n-1}$.

Let $N \in \FPinj{n-1}$ and consider the induced exact sequence 
\[
\Ext^1_R(F_0, N) \to \Ext^1_R(K, N) \to \Ext^2_R(M,N).
\]
Note that $\Ext^1_R(F_0, N) = 0$, since $F_0$ is projective. All we are left to do, to conclude that $K \in \FP{n-1}$, is to show that $\Ext^2_R(M,N) = 0$. Consider the following short exact sequence
\[
0 \to N \to I \to \Omega^{-1}(N) \to 0,
\] 
where $I$ is an injective module and $\Omega^{-1}(N)$ is the first cosyzygy of $N$. 

Now, by hypothesis, $M$ is left orthogonal to every FP$_n$-injective module, so if  we show that $\Omega^{-1}(N) \in \FPinj{n}$, then $0 = \Ext^1_R(M,\Omega^{-1}(N)) \cong \Ext^2_R(M,N)$.  Let $F \in \FP{n}$ and consider the following induced exact sequence:
\[
\Ext^1_R(F, I) \to \Ext^1_R(F, \Omega^{-1}(N)) \to \Ext^2_R(F,N).
\]
This time $\Ext^1_R(F, I) = 0$ since $I$ is injective. On the other hand, $\Ext^2_R(F,N) \cong \Ext^1_R(\Omega^{1}(F), N)$, with $\Omega^{-1}(F) \in \FP{n-1}$. But $N \in \FPinj{n-1}$, so we obtain $\Ext^1_R(\Omega^{1}(F), N) = 0$, and thus $\Ext^2_R(F,L) = 0$. This gives us that 
\[
\Ext^1_R(F,\Omega^{-1}(N)) = 0, 
\]
for every $F \in \FP{n}$, which says that $\Omega^{-1}(N) \in \FPinj{n}$. Hence $\Ext^2_R(M,N) = 0$. 

Finally $\Ext^1_R(K,N) = 0$ for every $N \in \FPinj{n-1}$, which along with the fact that $K \in \FP{n-2}$, and the induction hypothesis, gives us $K \in \FP{n-1}$. Therefore, $M \in \FP{n}$. 

The direction \eqref{M-is-in-FPn} implies \eqref{M-is-in-FP(n-1)-and-Ext1-is-0} is clear. 
\end{proof}

We also provide a version of the previous lemma for the class of $\text{FP}_n$-flat modules.

\begin{lemma} \label{characterization-of-FPn-by-Tor1} 
Let $n > 1$, then the following conditions are equivalent for every $M \in \modR$.
\begin{enumerate}
\item $M \in \FP{n-1}$ and $\Tor^R_1(N,M) = 0$ for every $N \in \FPflat{n}$. \label{M-is-in-FP(n-1)-and-Tor1-is-0}

\item $M \in \FP{n}$.  \label{M-is-in-FPn-Tor1}
\end{enumerate}
\end{lemma}

\begin{proof} 
The direction \eqref{M-is-in-FPn-Tor1}  implies \eqref{M-is-in-FP(n-1)-and-Tor1-is-0} is clear. 

Next, assume that $M$ satisfy \eqref{M-is-in-FP(n-1)-and-Tor1-is-0} and let $n > 1$. We will let $N \in \FPinj{n}$ and show that $\Ext^1_R(M,N) = 0$, since by Lemma~\ref{characterization-of-FPn-by-Ext1} we will get \eqref{M-is-in-FPn-Tor1}. For $n > 1$ we have $N^+ \in \FPflat{n}$, and so $\Tor^R_1(M , N^+) = 0$. On the other hand, by Theorem~\ref{Tor-Ext-relations} we have  that 
\[
\Tor^R_1(M, N^+) \cong \Ext^1_R(M, N)^+. 
\]
It follows $\Ext^1_R(M,N)^+ = 0$, and thus $\Ext^1_R(M,N) = 0$ for every $N \in \FPinj{n}$.

\end{proof}

\begin{remark}
Lemma \ref{characterization-of-FPn-by-Ext1} is stated for $n \geq 1$. This is not the situation for Lemma~\ref{characterization-of-FPn-by-Tor1}, since the duality relations do not allow us the case $n=1$.
\end{remark}

We are now ready to prove the main results of this section.

\begin{theorem}[Characterization of $n$-coherent rings via ${\text{FP}}_n$-injective modules] \label{characterization-of-n-coh-via-FPninj}
The following conditions are equivalent in $\Rmod$ for every $n \geq 1$.
\begin{enumerate}
\item $R$ is $n$-coherent. \label{R-is-n-coherent-with-her-cot-pair}

\item $\FPinj{n} = \FPinj{\infty}$. \label{FPninj=FPinftyinj}

\item $\FPinj{n} = \FPinj{n+1}$. \label{FPninj=FPn+inj}

\item $\FPinj{n}$ is coresolving. \label{FPninj-corresolving}

\item $(^\perp(\FPinj{n}), \FPinj{n})$ is a hereditary cotorsion pair. \label{FPninj-hereditary-cotorsion-pair}

\item $\Ext^k_R(M,N) = 0$ for every $k > 1$, for every $M \in \FP{n}$ and every $N \in \FPinj{n}$.  \label{higher-Ext-for-FPninj}
\end{enumerate}
\end{theorem}

\begin{proof} 
\

\eqref{R-is-n-coherent-with-her-cot-pair} $\Rightarrow$ \eqref{FPninj=FPinftyinj}: Suppose $R$ is $n$-coherent. Then $\FP{n} = \FP{\infty}$ by Theorem~\ref{characterization-of-n-coherent}, then $\FPinj{n} = \FP{n}^\perp = \FP{\infty}^\perp = \FPinj{\infty}$. 

\eqref{FPninj=FPinftyinj} $\Rightarrow$ \eqref{FPninj=FPn+inj}: If $\FPinj{n} = \FPinj{\infty}$, then the chain from \eqref{chain-of-FPn-inj} collapses at $n$, and so $\FPinj{n} = \FPinj{n+1}$.

\eqref{FPninj=FPn+inj} $\Rightarrow$ \eqref{FPninj-corresolving}: Under the assumption that $\FPinj{n} = \FPinj{n+1}$, we will show that $\FPinj{n+1}$ is coresolving, thus making $\FPinj{n}$ coresolving. Suppose we are given a short exact sequence $\ses{A}{B}{C}$ with $A, B \in \FPinj{n+1}$, and let $F \in \FP{n+1}$. Then we obtain the following induced exact sequence:
\[
\Ext^1_R(F,B) \to \Ext^1_R(F,C) \to \Ext^2_R(F,A).
\] 
Since $F \in \FP{n+1}$ and $B \in \FPinj{n+1}$, then $\Ext^1_R(F,B) = 0$. Also note that $\Ext^2_R(F,A) \cong \Ext^1_R(\Omega(F),A)$, where $\Omega(F)$ is the first syzygy of $F$, and as such $\Omega(F) \in \FP{n}$. But $A \in \FPinj{n+1} = \FPinj{n}$, so $\Ext^2_R(F,A)=\Ext^1_R(\Omega(F),A)=0$. This makes $\Ext^1_R(F,C)=0$ for any $F \in \FP{n+1}$, i.e. $C \in \FPinj{n+1}$, and therefore $\FPinj{n} = \FPinj{n+1}$ is coresolving. On the other hand, we already know that $\FPinj{n}$ is closed under extensions and contains that class of injective modules, giving us the result.

\eqref{FPninj-corresolving} $\Rightarrow$ \eqref{FPninj-hereditary-cotorsion-pair}: We know that $(^{\perp}(\FPinj{n}),\FPinj{n})$ is a cotorsion pair by Corollary \ref{cotorsion-pairs-FPninj-right}, so if $\FPinj{n}$ is coresolving, then  \cite[Lemma 2.2.10]{Gobel} says that the cotorsion pair $(^{\perp}(\FPinj{n}),\FPinj{n})$ is hereditary. 

\eqref{FPninj-hereditary-cotorsion-pair} $\Rightarrow$ \eqref{higher-Ext-for-FPninj}: We have that $\Ext^k_R(M,N) = 0$ for every $k > 1$, for every $M \in {}^\perp\FPinj{n}$ and every $N \in \FPinj{n}$, by definition. Hence the result follows since $\FP{n} \subseteq {}^\perp(\FPinj{n})$.  

\eqref{higher-Ext-for-FPninj} $\Rightarrow$ \eqref{R-is-n-coherent-with-her-cot-pair}: Let $M \in \FP{n}$, then there is a short exact sequence 
\[
\ses{K}{F}{M},
\]
with $K \in \FP{n-1}$ and $F$ finitely generated and free. Let $N \in \FPinj{n}$, and obtain the following induced exact sequence:
\[
\Ext^1_R(F,N) \to \Ext^1_R(K,N) \to \Ext^2_R(M,N).
\]
Note that $\Ext^1_R(F,N)=0$ since $F$ is projective, and that $\Ext^2_R(M,N)=0$ by hypothesis. Thus $\Ext^1_R(K,N)=0$ for all $N \in \FPinj{n}$. So, by Lemma~\ref{characterization-of-FPn-by-Ext1}, we have that $K \in \FP{n}$. This means that $M \in \FP{n+1}$, giving us the result.
\end{proof}

We also have a version of the previous theorem via $\text{FP}_n$-flat modules.

\begin{theorem}[Characterization of $n$-coherent rings via ${\rm FP}_n$-flat modules] \label{characterization-of-n-coh-via-FPnflat}
The following conditions are equivalent in $\modR$ for every $n > 1$.
\begin{enumerate}
\item $R$ is $n$-coherent. \label{R-is-n-coherent-with-her-cot-pair-flat}

\item $\FPflat{n} = \FPflat{\infty}$. \label{FPnflat=FPinftyflat}

\item $\FPflat{n} = \FPflat{n+1}$. \label{FPnflat=FPn+flat}

\item $\FPflat{n}$ is resolving. \label{FPnflat-resolving}

\item $(\FPflat{n}, (\FPflat{n})^\perp)$ is a hereditary cotorsion pair. \label{FPnflat-hereditary-cotorsion-pair}

\item $\Tor^R_1(N,M) = 0$ for every $N \in \FP{n}$ and $M \in \FPflat{n}$.  \label{higher-Tor-for-FPnflat}
\end{enumerate}
\end{theorem}

\begin{proof} The implications \eqref{R-is-n-coherent-with-her-cot-pair-flat} $\Rightarrow$ \eqref{FPnflat=FPinftyflat}, \eqref{FPnflat=FPinftyflat} $\Rightarrow$ \eqref{FPnflat=FPn+flat}, \eqref{FPnflat=FPn+flat} $\Rightarrow$ \eqref{FPnflat-resolving}, \eqref{FPnflat-resolving} $\Rightarrow$ \eqref{FPnflat-hereditary-cotorsion-pair}, and \eqref{higher-Tor-for-FPnflat} $\Rightarrow$ \eqref{R-is-n-coherent-with-her-cot-pair-flat} are dual to the corresponding implications in Theorem~\ref{characterization-of-n-coh-via-FPninj}. 

We only prove \eqref{FPnflat-hereditary-cotorsion-pair} $\Rightarrow$ \eqref{higher-Tor-for-FPnflat}. Let $N \in \FP{n}$ and $M \in \FPflat{n}$. Consider a short exact sequence
\[
0 \to K \to P \to M \to 0
\]  
where $P$ is projective. Then we the following induced exact sequence 
\[
\Tor^R_{k+1}(N,P) \to \Tor^R_{k+1}(N,M) \to \Tor^R_k(N,K) \to \Tor^R_k(N,P),
\]
where its ends are $0$ since $P$ is projective, and so it follows $\Tor^R_{k+1}(N,M) \cong \Tor^R_k(N,K)$. In particular, $\Tor^R_2(N,M) \cong \Tor^R_1(N,K)$. On the other hand, we are assuming that $(\FPflat{n}, (\FPflat{n})^\perp)$ is hereditary, and thus $K \in \FPflat{n}$. Then we have $\Tor^R_1(N,K) = 0$, and so $\Tor^R_2(N,M) = 0$. The rest follows by induction on $k > 1$. 
\end{proof}

Using these two results, we illustrate the chains in \eqref{chain-of-FPn-inj} with the ring in Example \ref{2-coherent-ring-example}.

\begin{example} \label{example-chains-stop-2} 
Consider the ring $R := k[x_1, x_2, x_3, \dots] / (x_i x_j)_{i,j \geq 1}$, with $k$ any field. Then we have the following chain of finitely generated modules:
\[
\FP{0}  \supset \FP{1} \supset \FP{2} = \FP{n} = \FP{\infty}.
\]
Recall that this ring is $2$-coherent and that $\FPinj{\infty} = \Rmod = \FPflat{\infty}$, i.e. every module is FP$_{\infty}$-injective and FP$_{\infty}$-flat. Furthermore, Theorem~\ref{characterization-of-n-coh-via-FPninj} and Theorem~\ref{characterization-of-n-coh-via-FPnflat}, tell us that the the following chains also collapse at $n = 2$,
\[
\FPinj{0}  \subset  \FPinj{1}  \subset  \FPinj{2} = \FPinj{\infty}, 
\]
\[
\FPflat{0}  \subset  \FPflat{1}  \subset \FPinj{2} = \FPflat{\infty}. 
\]
where the inclusions are strict. 

As a complement, we also exhibit a module that shows that that the inclusions $\FPinj{1} \subset \FPinj{\infty}$ and $\FPflat{1} \subset \FPflat{\infty}$ are indeed strict. We begin 
by looking %
for an FP$_\infty$-injective module which is not FP-injective. Consider the ideal $(x_1)$ as an  $R$-module over $R$. Then we have a short exact sequence 
\[
0 \to (x_1) \to R \to R / (x_1) \to 0,
\]
which is not split (if it is, then $x_1$ will act non trivially on the sum). So 
\[
\displaystyle \Ext^1_R(R/(x_1), (x_1)) \neq 0,
\]
and hence $(x_1)$ is not absolutely pure, since $R/(x_1)$ is finitely presented (as seen in Example \ref{2-coherent-ring-example}). Thus, $(x_1) \in \FPinj{\infty} \setminus \FPinj{1}$. On the other hand, the right $R$-module $(x_1)^+ \in \FPflat{\infty} \setminus \FPflat{1}$ exhibits that $\FPflat{1} \subset \FPflat{\infty}$ is indeed strict. 
\end{example}

%

\bibliographystyle{alpha}
\bibliography{bibliography-finiteness-cond}

\end{document}